\documentclass{amsart}
\usepackage{amssymb}
\usepackage{graphics}

\theoremstyle{plain}

\numberwithin{equation}{section}

\newtheorem{proposition}[equation]{Proposition}
\newtheorem{theorem}{Theorem}
\Huge
\newtheorem{lemma}[equation]{Lemma}

\theoremstyle{definition}
\newtheorem{remark}[equation]{Remark}
\newtheorem{example}[equation]{Example}

\def    \R  {{\Bbb R}}
\def    \Z  {{\Bbb Z}}

\def    \CP {{\Bbb {CP}}}

\def    \C  {{\Bbb C}}

\def    \Tilde  {\widetilde}

\def    \Gt     {\Tilde{G}}

\begin{document}
\title[K-contact manifolds with minimal closed Reeb orbits]{K-contact manifolds with minimal closed Reeb orbits}

\author{Hui Li} 
\address{School of mathematical Sciences\\
        Soochow University\\
        Suzhou, 215006, China.}
        \email{hui.li@suda.edu.cn}

\thanks{2010 classification. Primary:  53D10, 53D05, 53D20; Secondary: 55N10, 57R20}

\begin{abstract}
We use the Boothby-Wang fibration to construct certain simply connected K-contact manifolds and we give sufficient and necessary conditions on when such K-contact manifolds are homeomorphic to the odd dimensional spheres. If the symplectic base manifold of the fibration admits a Hamiltonian torus action,  we show that on the total space of the fibration, other than the regular K-contact structures which have infinitely many closed Reeb orbits, there are K-contact structures whose closed Reeb orbits correspond exactly to the fixed points of the Hamiltonian torus action on the base manifold. Then we give a collection of examples of compact simply connected K-contact manifolds with minimal number of closed Reeb orbits which are not homeomorphic to the odd dimensional spheres, while having the real cohomology ring of the spheres. Finally, we give a family of examples of simply connected K-contact manifolds which have one more than the minimal number of closed Reeb orbits and which do not have the real cohomology ring of the spheres.
\end{abstract}

 \maketitle

\section{introduction}
 
A contact manifold is a smooth manifold $M$ of odd dimension $2n+1$ with a contact 1-form $\alpha$ such that $\alpha\wedge(d\alpha)^n\neq 0$ everywhere. The Reeb vector field $R$ of $\alpha$ is the vector field characterized by $\alpha (R) =1$ and $d\alpha(R, \cdot)=0$; the flow of $R$ is called the contact flow or the Reeb flow. Hence the contact 1-form $\alpha$ defines a symplectic form $\omega = d\alpha$ on the $2n$-dimensional subbundle $\ker(\alpha)$. The contact manifold
$(M, \alpha)$ is called {\bf regular} if the Reeb vector field $R$ of $\alpha$ generates a free circle action. A Riemannian metric $g$ on $M$ is said to be {\bf adapted to $\alpha$} if $g$ is preserved by the Reeb flow of $\alpha$, and there exists an almost complex structure $J$ on $\ker(\alpha)$ such that $g|_{\ker(\alpha)}$, $J$ and $d\alpha$ are compatible, i.e., $g(X, Y) = d\alpha(X, JY)$ for all $X, Y$ in $\ker(\alpha)$. When there is an adapted Riemannian metric $g$ on $(M, \alpha)$, we call $(M, \alpha, g)$ a {\bf K-contact manifold}, and call $(\alpha, g)$ 
a K-contact structure on $M$. For a compact K-contact manifold $(M, \alpha, g)$, the closure of the Reeb flow in the isometry group of $M$ is a compact abelian Lie group, a $k$-dimensional torus $T^k$, with $1\leq k\leq n+1$ if $\dim(M)=2n+1$ (see \cite{{MS}, {R}}), and $k$ is called the {\bf rank} of the K-contact manifold $(M, \alpha, g)$.  Clearly the torus $T^k$ acts on the K-contact manifold $(M, \alpha, g)$ preserving the contact form $\alpha$ and the Riemannian metric $g$. On a smooth manifold $M$, there may exist different contact forms and K-contact structures with different ranks. For example, the sphere $S^{2n+1}$ with its standard contact form $\alpha_0$ induced from the standard symplectic form on $\C^{n+1}$ is a regular contact manifold; $\alpha_0$ and the standard metric $g_0$ on $S^{2n+1}$ give $S^{2n+1}$ a K-contact structure $(\alpha_0, g_0)$ of rank 1. A perturbation of $\alpha_0$ can give $S^{2n+1}$ a different contact form $\alpha$, which has an adapted Riemannian metric $g$, a perturbation of $g_0$. We may have a K-contact manifold $(S^{2n+1}, \alpha, g)$ of rank $k$ for any $2\leq k\leq n+1$, see Example~\ref{sphere}.

Recall that a {\bf contact toric manifold} is a contact manifold of dimension $2n+1$ with an effective 
contact $T^{n+1}$-action; it is called {\bf of Reeb type} if the Reeb vector field of the contact form is generated by a one parameter subgroup action of $T^{n+1}$.
We can equip contact toric manifolds of Reeb type with $T^{n+1}$-invariant Riemannian metrics adapted to the contact form, hence contact toric manifolds of Reeb type is a special subclass of K-contact manifolds.

In this paper, we are interested in working on a special class of K-contact manifolds which are principal circle bundles over symplectic manifolds with symplectic forms representing integral cohomology classes. The fact that these spaces possess contact structures is due to Boothby and Wang. Recall that a smooth fiber bundle $\pi\colon M\to N$ with fiber $S^1$ is called a  {\it principal} $S^1$-bundle if $S^1$ acts smoothly and freely on $M$ and the fiber preserving local trivialization maps are $S^1$-equivariant, where $S^1$ acts on local trivializations 
$U_{\alpha}\times S^1$ by multiplication on the $S^1$-factor, $\{U_{\alpha}\}$ being an open cover of $N$. It is known that principal circle bundles over a compact manifold 
$N$ are classified by elements in $H^2(N; \Z)$ through a characteristic class, called the Euler class of the circle bundle (\cite{K}). Here we are only concerned with principal circle bundles with Euler classes  in the image of the natural homomorphism $i\colon H^2(N; \Z)\to H^2(N; \R)$, namely the Betti part of $H^2(N; \Z)$. Elements in the image of $i$ are called {\bf integral classes}.
If $N$ is simply connected or $H_1(M; \Z)=0$ or $H_1(M; \Z)$ has no torsion, then $H^2(N; \Z)$ has no torsion by the universal coefficient theorem; in this case, $H^2(N; \Z)$ is isomorphic to its image in  $H^2(N; \R)$.  An integral  class  is called {\bf primitive} if it is not a positive integer multiple of any other integral class; for example, when $H^2(N; \Z)\cong \Z \hookrightarrow \R$, only the classes which are identified with $\pm 1$ are primitive integral. In our theorems, according to the necessity, 
we take integral or primitive integral classes in the degree 2 cohomology groups of the base manifolds of principal circle bundles.

Before the contact structures are very involved, we first make some claims on principal circle bundles over general compact manifolds, and  over compact symplectic manifolds with Hamiltonian circle actions.  These are our Theorems~\ref{thm1} and \ref{thm2} below. In this paper, when we say the integral (or real) cohomology group or the integral (or real) cohomology ring of a manifold, we mean the cohomology group or the cohomology ring of the manifold in integer (or real) coefficients. 

\begin{theorem}\label{thm1}
 Let $N$ be a compact smooth manifold of dimension $2n$,  
$x$ be a primitive integral degree 2 cohomology class of $N$, and $\pi\colon M\to N$ be a principal circle bundle over $N$ with Euler class $x$. Then the following claims hold.
\begin{enumerate}
\item $M$ has the real cohomology ring of $S^{2n+1}$ if and only if $N$ has the real cohomology ring of $\CP^n$, i.e., $H^*(N; \R)=\R[x]/x^{n+1}$ as a ring.
\item $M$ has the integral cohomology ring of $S^{2n+1}$ if and only if $N$ has the integral cohomology ring of $\CP^n$, i.e., $H^*(N; \Z)=\Z[x]/x^{n+1}$ as a ring.
\item $M$ is homeomorphic to a sphere $S^{2n+1}$ if and only if $\pi_1(N)=1$ and 
$H^*(N; \Z)=\Z[x]/x^{n+1}$ as a ring.
\end{enumerate}
\end{theorem}

Recall that if a compact symplectic manifold $(N, \omega)$ of dimension $2n$ admits a Hamiltonian circle action, then the action has at least $n+1$ fixed points. This is due to the fact that the moment map is a perfect Morse function, and $b_{2i}(N)\geq 1$ for any $0\leq 2i\leq 2n$. 
Now we take a compact symplectic manifold $(N, \omega)$ of dimension $2n$ as a base manifold in a principal circle bundle, then the Euler class of the principal circle bundle is the same as the first Chern class of the principal circle bundle.  We introduce a Hamiltonian circle action on $N$, then a condition on the base manifold is sufficient or sufficient and necessary for the total space to be homeomorphic to $S^{2n+1}$, this is Theorem~\ref{thm2}. In Theorem~\ref{thm2} (2), $c_1(\cdot)$ and $c(\cdot)$ are respectively the first Chern class and total Chern class of the corresponding space.

\begin{theorem}\label{thm2}
Let $(N, \omega)$ be a compact symplectic manifold of dimension $2n$ with $[\omega]$ being a primitive integral cohomology class, and let $\pi\colon M\to N$ be a principal circle bundle with first Chern class $x= [\omega]$. Assume $N$ admits a Hamiltonian circle action with fixed point set consisting of $n+1$ isolated points. Then the following claims hold.
\begin{enumerate}
\item $M$ is homeomorphic to a sphere $S^{2n+1}$ if and only if $H^*(N; \Z)=\Z[x]/x^{n+1}$ as a ring. 
\item If $c_1(N)=(n+1)x$, then $c(N)\cong c(\CP^n)=(1+x)^{n+1}$, and $M$ is homeomorphic to a sphere $S^{2n+1}$. 
\end{enumerate}
\end{theorem}

In Theorem~\ref{thm2}, if there is a Hamiltonian torus action on $N$, then there is a Hamiltonian subcircle action with the same fixed point set as the torus action.  Hence for the claims in Theorem~\ref{thm2}, assuming a circle action on $N$ is more general than assuming a torus action.
See Remark~\ref{tc}.

In Theorem~\ref{thm2} (2), for $2\leq\dim(N)=2n\leq 8$, 
$c_1(N)=(n+1)x$ is also a necessary condition for $M$ to be homeomorphic to $S^{2n+1}$.
This can be seen as follows. By Theorem~\ref{thm1} (3) or Theorem~\ref{thm2} (1), $M$ is homeomorphic to $S^{2n+1}$ implies $H^*(N; \Z)=\Z[x]/x^{n+1}$ as a ring. 
Under the assumption in Theorem~\ref{thm2} for the manifold $N$, for $2\leq\dim(N)=2n\leq 8$, when $H^*(N; \Z)=\Z[x]/x^{n+1}$ as a ring, known results tell us that $c(N)\cong c(\CP^n)$, in particular, $c_1(N)=(n+1)x$ (see the Introduction of \cite{L'} and the references therein).
In general case of Theorem~\ref{thm2} (2), the author does not know if $c_1(N)=(n+1)x$  is a necessary condition for $M$ to be homeomorphic to $S^{2n+1}$. 
In the special case when there is a maximal dimensional symmetry, namely there is a Hamiltonian $T^n$-action on $N$ with exactly $n+1$ isolated fixed points, by our Theorem~\ref{thm1} (3) and Petrie's work (\cite{P}), we have the isomorphism 
$c(N)\cong c(\CP^n)=(1+x)^{n+1}$, in particular $c_1(N)=(n+1)x$. Or, solely for a symplectic manifold $N$ of dimension $2n$, when there is a Hamiltonian $T^n$-action on $N$ with exactly $n+1$ isolated fixed points, by Jang's work in \cite{J} for almost complex torus manifolds, we have $c(N)\cong c(\CP^n)$, in particular $c_1(N)=(n+1)x$ (independent of what $M$ is). 

In Theorem~\ref{thm2}, we assume that the fixed points of the Hamiltonian circle action on $N$ are isolated, this is what we concern in the current work. When the fixed points of the Hamiltonian circle action on $N$ are not isolated, we may still have the claims (1) and (2) of the theorem, and in (2), in certain case, the fact that $M$ is homeomorphic to $S^{2n+1}$ implies $c(N)= c(\CP^n)$.
See Remark~\ref{niso}.

Next, the contact and K-contact structures on the total space of a principal circle bundle over a symplectic base manifold are  more involved.
Let $(N, \omega)$ be a compact symplectic manifold of dimension $2n$ with $[\omega]$ being an integral cohomology class, and let $\pi\colon M\to N$ be a principal circle bundle with first Chern class $[\omega]$.
By Boothby and Wang  (\cite{BW}), there is a connection 1-form $\alpha$ on $M$ such that $d\alpha= \pi^*(\omega)$,  $(M, \alpha)$ is a compact contact manifold, and the Reeb flow of $\alpha$ generates a free circle action which is exactly the principal circle action of the circle bundle. 
In this case, we call the principal circle bundle $\pi\colon M\to N$ a {\bf Boothby-Wang fibration}.
For any Riemannian metric $g_N$ on $N$ compatible with $\omega$, $g=\pi^*(g_N) + \alpha\otimes\alpha$ defines an adapted metric on $M$ (see Section 4), so $(M, \alpha, g)$ becomes a K-contact manifold; it is of rank one since each Reeb orbit of $\alpha$ is closed. Note that we can have many different choices of $g_N$, hence of $g$.
Assume now that  $N$ admits a Hamiltonian torus $T$ action, we will see that we can lift the $T$-action on $N$ to a 
$T$-action on $M$ such that it preserves the connection one form $\alpha$ and the metric $g=\pi^*(g_N) + \alpha\otimes\alpha$ if $g_N$ is $T$-invariant on $N$. Let $S^1$ be the fiber transformation group of the principal circle bundle $\pi\colon M\to N$. Then the $T\times S^1$-action preserves $\alpha$ and $g$ on $M$, so $T\times S^1$ is in the isometry group of the K-contact manifold $(M, \alpha, g)$.  Exactly this more symmetry on $M$ will give us the room to vary the K-contact structures on $M$. We will show that 
other than the regular K-contact structures above which have infinitely many closed Reeb orbits, for each $2\leq k\leq \dim(T\times S^1)$,
$M$ admits a K-contact structure of rank $k$ such that its closed Reeb orbits are in one to one correspondence with the fixed points of the Hamiltonian $T$-action on $N$, which are finite. This is our Theorem~\ref{thm3} below.

\begin{theorem}\label{thm3}
Let $(N, \omega)$ be a compact symplectic manifold of dimension $2n$ 
with $[\omega]$ being an integral class, and let $\pi\colon M\to N$ be a principal circle bundle with first Chern class $[\omega]$. Assume $N$ admits a Hamiltonian torus $T^l$-action with $1\leq l\leq n$, then for any $2\leq k\leq l+1$, there exists a K-contact structure on $M$ of rank $k$ whose closed Reeb orbits are in one to one correspondence with the fixed points of the $T^l$-action on $N$. If $(N, \omega)$ is a K\"ahler manifold, then  these K-contact structures on $M$  are Sasakian structures. 
\end{theorem}

In the Boothby-Wang fibration $\pi\colon M\to N$, for the regular contact form $\alpha$ with
$d\alpha = \pi^*(\omega)$, the quotient space by the Reeb flow is just the smooth manifold $N$. For each K-contact form in the conclusion of Theorem~\ref{thm3}, the closure of the Reeb flow is $T^k$, the quotient space by the Reeb flow is not a Hausdorff space, it is very singular --- it is locally the quotient of a Cartesian space by an affine action of a countable group; for a general description of such a quotient space, see \cite[Theorem 1.3]{LM}.

Theorem~\ref{thm3} is for Hamiltonian $T$ base manifold with general fixed point sets. 
It obviously applies to more general situations than the case to be mentioned below.
In the case that the fixed point set of the Hamiltonian $T$-action on $N$ consists of $n+1$ isolated points, by Theorem~\ref{thm3},  the $2n+1$-dimensional compact manifold $M$ admits  K-contact structures with exactly $n+1$ closed Reeb orbits. Similar to the number of fixed points for Hamiltonian $S^1$-actions on compact symplectic manifolds,  a compact K-contact manifold of dimension $2n+1$ has at least $n+1$ number of closed Reeb orbits, see \cite{R0} or \cite{GNT} for different proofs.

There is the claim by Rukimbira that compact  K-contact manifolds of dimension $2n+1$ with exactly $n+1$ closed Reeb orbits are finitely covered by $S^{2n+1}$, and homeomorphic to $S^{2n+1}$ if they are simply connected, see \cite{{R1}, {R2}} (also quoted in the book \cite[Theorem 7.4.7]{BG}). In \cite{GNT}, Goertsches, Nozawa and T\"oben constructed a $7$-dimensional simply connected counter example of the claim, the Stiefel manifold $V_2(\R^5)=SO(5)/SO(3)$. We observe that the manifold $V_2(\R^5)$ happens to be a principal circle bundle over the Grassmannian manifold of oriented 2-planes in $\R^5$. This fact motivates the  
author to use the current method to look at the problem. 
We use the current method to show that  Goertsches-Nozawa-T\"oben's example belongs to a  general family of counter examples, the Stiefel manifold $V_2(\R^{n+2})=SO(n+2)/SO(n)$, which is a principal circle bundle over the Grassmannian manifold of oriented 2-planes in $\R^{n+2}$ for any $n\geq 3$ odd, see Example~\ref{exgrass}. We will use Theorems~\ref{thm2} and \ref{thm3} to give more examples of compact simply connected K-contact manifolds of dimension $2n+1$ with exactly $n+1$ closed Reeb orbits which are not homeomorphic to $S^{2n+1}$. This is our Theorem~\ref{thm4} below.

\begin{theorem}\label{thm4}
 There are simply connected compact K-contact manifolds of dimension $2n+1$ with exactly $n+1$ isolated closed Reeb orbits which are not homeomorphic to $S^{2n+1}$, while having the real cohomlogy ring of $S^{2n+1}$. See Section 7 for the examples constructed.
\end{theorem}

We hope that the observation of this paper gives a relatively full account of the problem involved, especially for the case of minimal isolated Reeb orbits for compact K-contact manifolds.
The current idea and method can provide a bridge and a view for understanding Hamiltonian dynamics on symplectic manifolds and on contact flows on contact manifolds. 
In this paper, we are mainly concerned with the closure of contact flows, hence we only consider Hamiltonian $T$-actions on the symplectic base manifold. In other contexts, one may consider more general Hamiltonian Lie group actions on the symplectic base manifold.
We hope that the idea and method and Theorem~\ref{thm3} will be used more broadly for symplectic base manifold $N$ with Hamiltonian Lie group actions with more fixed points --- isolated or non-isolated.

We mention here in \cite{L23}, for compact contact toric manifolds of Reeb type, a subclass of K-contact manifolds of maximal rank, in the regular contact case,  the author gives sufficient conditions for the manifolds to be homeomorphic to the spheres. In \cite{L23}, we use Boothby-Wang's claim that the quotient space of a regular contact manifold is a symplectic manifold. In this paper, we use Boothby-Wang's converse claim that a principal circle bundle over a symplectic manifold with a symplectic form representing an integral cohomology class is a contact manifold. The current results are for general K-contact manifolds.

Now we summarize the structure of the paper. In Section 2, we prove Theorem~\ref{thm1}.
In Section 3, we prove Theorem~\ref{thm2}. In Section 4, we recall the Boothby-Wang fibration, and state the  contact  and  K-contact structures on the total space. In Section 5, we lift the Hamiltonian torus action on the symplectic base manifold to the total space of the Boothby-Wang fibration. In Section 6, we perturb the K-contact structures on the total space of the Boothby-Wang fibration and prove Theorem~\ref{thm3}. In Section 7, we use Theorems~\ref{thm2} and \ref{thm3}  to construct the collection of examples of simply connected compact K-contact manifolds of dimension $2n+1$ with exactly $n+1$ closed Reeb orbits which are not homeomorphic to $S^{2n+1}$, but are real cohomology spheres, proving Theorem~\ref{thm4}. In Section 8, through the Boothby-Wang fibration, we give a family of examples of simply connected K-contact manifolds of dimension $2n+1$ which are not real cohomology spheres, and admit K-contact structures with more than $n+1$ number of closed Reeb orbits, using our Theorem~\ref{thm3}.
  
\subsubsection*{Acknowledgement}  I thank the anonymous referee for reading this paper and for making some comments which are helpful to improve the exposition.

\section{proof of Theorem~\ref{thm1}}
In this section, we give a proof  of  Theorem~\ref{thm1}. 
We first prove a lemma on the fundamental group of the total space of a principal circle bundle.
In this section, we only use Lemma~\ref{pi1m} (1), we will later use Lemma~\ref{pi1m} (2).

\begin{lemma}\label{pi1m}
Let $N$ be a connected compact manifold with a degree 2 primitive integral cohomology class $x$, and let $\pi\colon M\to N$ be a principal circle bundle over $N$ with Euler class $x$. Then the following claims hold.
\begin{enumerate}
\item If $\pi_1(N)=1$ and $H^2(N; \Z)=\Z$, then $\pi_1(M)=1$. 
\item If $N$ is symplectic with symplectic class $x$,  
$H^2(N; \Z)=\Z$, and $N$ admits a Hamiltonian circle action with  isolated fixed points, then $\pi_1(M)=1$.
\end{enumerate}
\end{lemma}

\begin{proof}
(1) Consider the homotopy exact sequence for the circle bundle $\pi\colon M\to N$:
\begin{equation}\label{hmtp}
\cdots\to \pi_2(N)\to\pi_1(S^1)\to\pi_1(M)\to\pi_1(N)\to\pi_0(S^1)\to\cdots.
\end{equation}
Since $\pi_1(N) = 1$, we have $\pi_2(N)= H_2(N; \Z)$, and $H^2(N; \Z)=H_2(N; \Z)$ (by the universal coefficient theorem).
By considering the universal $S^1$-bundle $ES^1\to BS^1$, a classifying map $N\to BS^1$, and a diagram chasing, we see that the map 
$$\pi_2(N)= H_2(N; \Z) = H^2(N; \Z)=\Z\to\pi_1(S^1) = H_1(S^1;\Z) = H^1(S^1; \Z)=\Z$$
in (\ref{hmtp}) is multiplication by the Euler class $x\in H^2(N; \Z)=\Z$ of the $S^1$-bundle $\pi\colon M\to N$.
Since $x$ is primitive integral, (up to a sign) the above map is multiplication by $1$, so the map
$\pi_2(N)\to\pi_1(S^1)$ is surjective, by the exact sequence  (\ref{hmtp}), we get
$$\pi_1(M) = 1.$$

(2) If $N$ is a connected compact symplectic manifold which admits a Hamiltonian circle action with isolated fixed points, then the minimum of the moment map is a single point, by \cite[Theorem 0.1]{L0}, $\pi_1(N)=\pi_1(\mbox{minimum})=1$,  so the claim follows from (1).
\end{proof}

\begin{remark}
In the proof of Lemma~\ref{pi1m}, we need the map $\pi_2(N)\to\pi_1(S^1)$ to be surjective to claim
$\pi_1(M) = 1$. 
The condition $H^2(N; \Z)=\Z$ is one such that it holds. One may have other conditions for the claim
to hold. 
\end{remark}

\begin{proof}[Proof of Theorem~\ref{thm1}]
(1) Suppose we have the isomorphism $H^*(N; \R)=\R[x]/x^{n+1}$ as a ring. Then $H^{\mbox{odd}}(N; \R)=0$. So the Gysin long exact sequence for the circle bundle $\pi\colon M\to N$ in $\R$-coefficients splits into a short exact sequence
$$0\to H^{2i+1}(M; \R)\overset{\pi_*}\longrightarrow H^{2i}(N; \R)\overset{f}\longrightarrow H^{2i+2}(N; \R)\overset{\pi^*}\longrightarrow H^{2i+2}(M; \R)\to 0,$$
where $\pi_*$ is integration along the fiber, $f$ is multiplication by the Euler class $x$ of the bundle $\pi\colon M\to N$, and $\pi^*$ is pull back. The ring isomorphism $H^*(N; \R)=\R[x]/x^{n+1}$ implies that $f$ is an isomorphism for each $0\leq 2i\leq 2n-2$, and is $0$ when $2i=2n$, so $H^*(M; \R)\cong H^*(S^{2n+1}; \R)$ as rings.

Conversely, if $H^*(M; \R)\cong H^*(S^{2n+1}; \R)$ as rings, then $H^k(M; \R)=0$ for $1\leq k\leq 2n$ and $H^{2n+1}(M; \R)=\R$. The long Gysin exact sequence
$$\cdots \to H^{*+1}(M; \R)\to H^*(N; \R)\to H^{*+2}(N; \R)\to H^{*+2}(M; \R)\to \cdots$$
implies that $f\colon H^*(N; \R)\to H^{*+2}(N; \R)$ is an isomorphism for each $0\leq *\leq 2n-1$, and is $0$ when $ *=2n$, hence $H^{2i}(N; \R)=\R$ for all $0\leq 2i\leq 2n$, and 
$H^{\mbox{odd}}(N; \R)=0$, so we have the ring isomorphism 
$H^*(N; \R)=\R[x]/x^{n+1}\cong H^*(\CP^n; \R)$.

(2)  Consider the long Gysin exact sequence for the circle bundle $\pi\colon M\to N$ in $\Z$-coefficients: 
$$\cdots \to H^{*+1}(M; \Z)\to H^*(N; \Z)\to H^{*+2}(N; \Z)\to H^{*+2}(M; \Z)\to \cdots$$
and follow the arguments in (1). In the process of the arguments, we need to note that $x\in H^2(N; \Z)$ implies that $x^i$  is an integral class for each $2\leq i\leq n$, and 
$$H^{2i}(N; \Z)\overset{f}\longrightarrow H^{2i+2}(N; \Z)$$
which is multiplication by $x$, is an isomorphism if and only 
if $x^i$ and $x^{i+1}$ are generators for  the two corresponding cohomology groups, where 
$0\leq i\leq n-1$.

(3) If $M$ is homeomorphic to $S^{2n+1}$, then $\pi_1(M)=1$ and $H^*(M; \Z)\cong H^*(S^{2n+1}; \Z)$ as rings; by (\ref{hmtp}), $\pi_1(N)=1$, and by (2), $H^*(N; \Z)=\Z[x]/x^{n+1}$ as a ring.
Conversely, if $\pi_1(N)=1$ and $H^*(N; \Z)=\Z[x]/x^{n+1}$ as a ring, then by Lemma~\ref{pi1m} (1), $\pi_1(M)=1$, and by (2), $H^*(M; \Z)\cong H^*(S^{2n+1}; \Z)$ as rings.  
By the generalized Poincar\'e conjecture  (\cite{{KL}, {S}}), $M$ is homeomorphic to $S^{2n+1}$.
\end{proof}

\section{proof of Theorem~\ref{thm2}}

In this section, we first prove Theorem~\ref{thm2}, then we make a few very helpful remarks for understanding the theorem.

\begin{proof}[Proof of Theorem~\ref{thm2}]
 (1)  When the compact symplectic manifold $N$ admits a Hamiltonian $S^1$-action with only isolated fixed points, by \cite[Theorem 0.1]{L0}, $\pi_1(N) =1$. Then the claim follows from Theorem~\ref{thm1} (3).

(2) When $N$ admits a Hamiltonian $S^1$-action with fixed point set  consisting of $n+1$ isolated points, if $c_1(N)=(n+1)x$, then by \cite[Theorem 1]{L}, the largest weight of the $S^1$-action at all the fixed points on $N$ is
equal to the length of the moment map image interval (which is an integer since $[\omega]$ is an integral cohomology class).  In this case, by \cite[Theorem 2]{L}, 
$H^*(N; \Z)=\Z[x]/x^{n+1}$ holds as a ring, and we also have the total Chern class isomorphism
$c(N)\cong c(\CP^n) = (1+x)^{n+1}$. By (1),  $M$ is homeomorphic to $S^{2n+1}$.
\end{proof}

\begin{remark}
In Theorem~\ref{thm2}, since the fixed point set of the Hamiltonian circle action consists of exactly $n+1$ isolated points, 
using the fact that the moment map is a Morse function on $N$, we obtain the integral cohomology groups of $N$ as follows (see \cite{L}):
$$H^{2i}(N; \Z)=\Z, \,\,\,\mbox{and}\,\,\, H^{2i+1}(N; \Z)=0 \,\,\,\mbox{for all $0\leq i\leq n$}.$$
The fact that $x=[\omega]$ is primitive integral implies that $x^i$ is an integral class and $0\neq x^i\in H^{2i}(N; \Z)$ since $N$ is compact symplectic, where $0\leq i\leq n$. 
So there exist natural numbers $a_2, \cdots, a_n$ such that $0\neq \frac{1}{a_i}x^i$ is primitive integral for $2\leq i\leq n$. So the integral cohomology ring  $H^*(N; \Z)$ is freely generated by 
$$1, \,\,\, x, \,\,\, \frac{1}{a_2}x^2, \,\,\,\cdots, \,\,\, \frac{1}{a_{n-2}}x^{n-2}, \,\,\, \frac{1}{a_{n-1}}x^{n-1}, \,\,\, \frac{1}{a_n}x^n.$$
By Poincar\'e duality, $a_ia_{n-i}=a_n$ for all $1\leq i\leq n-1$, where we regard $a_1=1$.
If we do not have the ring isomorphism $H^*(N; \Z)\neq\Z[x]/x^{n+1}$, i.e., the generators of 
the groups $H^{2i}(N; \Z)=\Z$'s are not as
$$1, \,\, x, \,\, \cdots,  x^i, \, \, \cdots, \,\, x^n,$$ 
then there is the first $a_k$ such that $a_k\neq 1$, and by the Gysin exact sequence of the principal circle bundle in $\Z$-coefficients:
$$\cdots\to H^{2k-2}(N; \Z)\to H^{2k}(N; \Z)\to H^{2k}(M; \Z)\to H^{2k-1}(N; \Z)=0 \to \cdots,$$
we get
$$H^{2k}(M; \Z) = \Z_{a_k},$$
a torsion group. In this case $M$ cannot be homeomorphic to $S^{2n+1}$.
\end{remark}

\begin{remark}
In Theorem~\ref{thm2}, if we assume that the fixed point set of the Hamiltonian circle action on $N$ consists of isolated points but more than $n+1$ isolated points, then $N$ has more even Betti numbers than $\CP^n$ implied by the fact that the moment map is a perfect Morse function; in this case, $M$ does not even have the real cohomology ring of a sphere by Theorem~\ref{thm1} (1). See Example~\ref{evengrass}.
\end{remark}

\begin{remark}\label{niso} 
In Theorem~\ref{thm2}, if we do not assume that the fixed points of the Hamiltonian circle action on $N$ are isolated, we may still have the claims (1) and (2). For example, when the fixed point set of the circle action consists of two connected components $X$ and $Y$ with  $\dim(X)+\dim(Y)+2=\dim(N)$,  \cite{LT} shows that this condition is equivalent to the condition that the even Betti numbers of $N$ are minimal, i.e., $b_{2i}(N) =1$ for each $0\leq 2i\leq 2n$. In this case, by \cite{LT}, $N$ either has the integral cohomology ring of $\CP^n$ or it has the integral cohomology ring of the Grassmannian of oriented 2-planes in $\R^{n+2}$, denoted $\Gt_2(\R^{n+2})$ with $n\geq 3$ odd.  By \cite{LOS}, the manifold $N$ is simply connected. In this case, we still have the claims 
(1) and (2) of Theorem~\ref{thm2}, moreover, in (2), the fact that $M$ is homeomorphic to $S^{2n+1}$ implies $c(N)\cong c(\CP^n)$ by (1) and the results in \cite{LT}.
\end{remark}

\begin{remark}\label{tc}
In Theorem~\ref{thm2}, if there is a Hamiltonian torus action on $N$, since $N$ is compact, there is a finite number of stabilizer groups for the torus action, so we can find a subcircle action with exactly the same fixed point set as the torus action. Hence for the claims in Theorem~\ref{thm2}, assuming a torus action is a stronger assumption than assuming a circle action on $N$, or assuming a circle action on $N$ is more general. 
\end{remark}

\section{K-contact structures on the total spaces of principal circle bundles over symplectic manifolds}

Suppose $(M, \alpha)$ is a compact regular contact manifold. Then by Boothby and Wang \cite{BW},  $M$ is a principal circle bundle over the Reeb orbit space $N=M/S^1$, $\alpha$ is a connection form in this bundle, and the curvature form of $\alpha$ projects to a symplectic form $\omega$ on $N=M/S^1$ which represents an integral cocycle, i.e., $d\alpha = \pi^*(\omega)$, where $\pi\colon M\to N$ is the projection, and $[\omega]\in H^2(N; \Z)$.

Conversely, given a compact symplectic manifold $(N, \omega)$, with $[\omega]$ an integral cohomology class, Boothby and Wang also have a converse claim as follows, this is the part we will use in this paper.

\begin{theorem}(Boothby-Wang)\label{BW1}\cite{BW}
Let $(N, \omega)$ be a compact symplectic manifold with symplectic class $[\omega]$ being integral. Let 
$\pi\colon M\to N$ be a principal circle bundle over $N$ with first Chern class $[\omega]$, and $\alpha$ be a connection form on $M$ such that $d\alpha=\pi^*(\omega)$.
Then $(M, \alpha)$ is a contact manifold with regular contact form $\alpha$ whose Reeb flow generates  exactly the transformation group of the fibers of the principal circle bundle.
\end{theorem}

In the Boothby-Wang fibration in Theorem~\ref{BW1} with connection $\alpha$, suppose $g_N$ is any Riemannian metric on $N$ compatible with $\omega$, i.e., there is an almost complex structure $J_N$ on $N$ such that 
$g_N (X, Y)=\omega(X, J_N Y)$ for any $X, Y\in TN$, we define an almost complex structure $J_M$ on 
$\ker(\alpha) \subset TM$ by
$$J_M X = \big(J_N(\pi_* X)\big)_{\mbox{hor}},\quad \forall\,\, X\in \ker(\alpha),$$
where $\big(J_N(\pi_* X)\big)_{\mbox{hor}}$ is the horizontal lift of the vector field $J_N(\pi_* X)$.  
We have $\pi_*(J_M X) = J_N(\pi_* X)$.
Then $g=\pi^*(g_N)\oplus \alpha\otimes\alpha$ is a Riemannian metric on $M$ invariant under the Reeb flow of $\alpha$, and is adapted to $\alpha$,  so  $(M, \alpha, g)$ is a K-contact manifold. 
Note that we may have many different choices for $g_N$ and hence for $g$. 
Conversely, let $g$ be an adapted Riemannian metric on $M$ to $\alpha$. Let $L_R$ be the Lie derivative with respect to the Reeb vector field $R$ of $\alpha$. Since $L_Rg=0$,
we have $L_RJ_M=0$, where $J_M$ is a compatible almost complex structure on $\ker(\alpha)$ (see \cite[Lemma 3.2]{YK}), so both $g$ and $J_M$ project respectively to a metric $g_N$ and 
an almost complex structure $J_N$ on $N$ such that $g_N$, $J_N$ and $\omega$ are compatible.
So $g$ is of the form $g=\pi^*(g_N)\oplus \alpha\otimes\alpha$, if we require that the Reeb vector field has norm 1 at every point (this is what is usually required for a K-contact structure).

Now let $(P, \eta, h)$ be a general K-contact manifold (or a contact metric manifold), where $P$ is a manifold, $\eta$ is a contact one form, and $h$ is an adapted metric. 
In the text books (see e.g. \cite{YK}), the almost complex structure $J$ is defined on the whole $TP$  by letting $J$ on $\ker(\eta)$ as we defined and letting 
$$J \xi =0,$$ 
where $\xi$ is the Reeb vector field of $\eta$. For any $X\in TP$, we can write $X=X_1 + \eta(X) \xi$ for some $X_1\in\ker(\eta)$. Then
$JX=JX_1$ and $J^2 (X) = J^2 X_1 = -X_1 = -X + \eta(X) \xi$, this is often written as
$$J^2 = - \mbox{Id} + \eta \otimes \xi.$$
 Let the Nijenhuis tensor $N_J$ be
$$N_J (X, Y) = J^2[X, Y] + [JX, JY] - J [JX, Y]-J [X, JY], \,\,\,  \forall \,\,\, X, Y \in TP,$$ 
where $[\cdot, \cdot]$ is the Lie bracket of vector fields.
If 
$$N_J + 2d\eta\otimes \xi = 0,$$
then we call the K-contact manifold $(P, \eta, h)$ a {\bf Sasakian manifold}, and the structure $(\eta, h)$
a Sasakian structure on $P$ (e.g. see Chapter V in \cite{YK}). We may view Sasakian manifolds as the odd dimensional analog of K\"ahler manifolds.

For the Boothby-Wang fibration $\pi\colon M\to N$ and the K-contact structure $(\alpha, g)$ on $M$
we constructed above, we extend $J_M$ to $TM$ by letting $J_M R=0$, where $R$ is the Reeb vector field of $\alpha$. Then by Hatakeyama \cite{H} and Morimoto \cite{M}, $(M, \alpha, g)$ is a Sasakian  manifold if and only if $(N, \omega, g_N)$ is a K\"ahler manifold. May see p290-291 in \cite{YK}, where this is checked. We summarize all the facts above as follows.

\begin{theorem}\cite{{YK}, {H}, {M}}\label{BW2}
In the Boothby-Wang fibration $\pi\colon M\to N$ in Theorem~\ref{BW1} with connection $\alpha$ such that $d\alpha=\pi^*(\omega)$, any K-contact Riemannian metric on $M$ adapted to $\alpha$ is of the form $g=\pi^*(g_N)\oplus \alpha\otimes\alpha$, where $g_N$ is any Riemannian metric on $N$ compatible with $\omega$. So $(M, \alpha, g)$ is a K-contact manifold. Moreover, $(M, \alpha, g)$ is a Sasakian manifold if and only if $(N, \omega, g_N)$ is a K\"ahler manifold.
 \end{theorem}

\section{lift a Hamiltonian torus action on the symplectic base to the total space of the fibration $\pi\colon M\to N$}

In this section, assuming that there is a Hamiltonian torus $T$-action on the compact base symplectic manifold $(N, \omega)$, where $[\omega]$ is an integral class, we show how to lift the $T$-action to the total space $M$ of the principal circle bundle $\pi\colon M\to N$, so there are more symmetries on $M$ than the principal circle action. 

\medskip

Recall that if $(M, \alpha)$ is a contact manifold with a compact Lie group $G$-action, then there exists a {\bf contact moment map}  $\psi\colon M \to \mathfrak g^*$ defined by
\begin{equation}\label{contactmoment}
\psi^X=\langle\psi, X\rangle = \alpha(X_M), \quad\forall \,\, X\in \mathfrak g=\mbox{Lie}(G),
\end{equation}
where $X_M$ is the vector field generated by the $X$-action on $M$.

Let $(N, \omega)$ be a compact symplectic manifold with integral symplectic class $[\omega]$. Let $\pi\colon M\to N$ be  a principal circle bundle with connection form $\alpha$ such that $d\alpha=\pi^*(\omega)$, so $(M, \alpha)$ becomes a compact contact manifold. For the symplectic manifold $(N, \omega)$, we assume that there is a Hamiltonian Lie group $G$-action  with moment map $\phi\colon N\to \mathfrak g^*$ satisfying
$$d\phi^X = d\langle\phi, X\rangle = - i_{X_N}\omega,   \quad\forall \,\, X\in \mathfrak g,$$
where $X_N$ is the vector field generated by the $X$-action on $N$.
It is known that the Lie algebra $\mathfrak g=\,$Lie$(G)$-action on $N$ lifts to the total space $M$ by the formula
\begin{equation}\label{lift}
 X_M = \Tilde {X_{N}} + \phi^X\frac{\partial}{\partial \theta},\quad \forall\,\, X\in\mathfrak g,
\end{equation}
where $\Tilde {X_{N}}$ is the horizontal lift of $X_N$ to $M$, and $\frac{\partial}{\partial \theta}$ is the vector field generated by the principal circle action on $M$, it is also the Reeb vector field of the contact 1-form $\alpha$. For this lift, one may refer to Chapter 6 of the book by Ginzburg, Guillemin and Karshon \cite{GGK} or the works
\cite{{Ki1}, {Ki2}, {Ko}} by Kostant and Kirillov.
We can check that $X_M$ commutes with $\frac{\partial}{\partial \theta}$ for all $X\in\mathfrak g$, and (\ref{lift}) gives a Lie algebra action of $\mathfrak g$ on $M$, i.e., 
$$[X_M, Y_M] = \Tilde{[X, Y]_N} + \phi^{[X, Y]}\frac{\partial}{\partial \theta},\quad \forall\,\, X, Y\in\mathfrak g.$$
Note that since $\alpha$ is the connection form, also the contact form on $M$, by (\ref{contactmoment}), the lifted Lie algebra action on $M$ produces a  contact moment map 
$$\psi^X = \alpha(X_M) = \phi^X,  \quad \forall\,\, X\in\mathfrak g.$$
This can be expected, since the contact moment map defined by the contact form in (\ref{contactmoment}) is the same as the moment map for the presymplectic form $\pi^*(\omega)$ on the presymplectic manifold $(M, \pi^*(\omega))$. 

The lift (\ref{lift}) uses the connection form to define the horizontal lift $\Tilde {X_{N}}$.
Two different connections $\alpha$ and $\alpha'$ such that $d\alpha=\pi^*(\omega)$ and $d\alpha' =\pi^*(\omega)$ differ by the pull back by $\pi^*$ of an invariant closed 1-form on $N$.
 In certain situations, the lift does not depend on the choice of the connection, it is uniquely determined by the symplectic form $\omega$  and the moment map $\phi$ on $N$, for example, when a maximal torus of $G$ has a fixed point on $N$, or when $N$ is compact and $H^1(N, \R)=0$. The reader may check this, or refer to page 101 in \cite{GGK}.  We do not use this point in this paper, we just need to choose a connection $\alpha$ such that $d\alpha=\pi^*(\omega)$.

In this paper, our concern is the closures of the contact flows, hence we are only concerned about  torus actions on $N$. We now assume that there is an effective Hamiltonian torus $T$ action on the compact base symplectic manifold $N$. It is known that $\dim(T)\leq \frac{1}{2}\dim(N)$.
Since $N$ is compact, the fixed point set of the Hamiltonian $T$-action on $N$ is not empty. The moment map $\phi$ of the $T$-action on $N$ is defined up to translation in $\mathfrak t^*$ by a constant, so we may assume that  at a fixed point of $T$, $\phi$ achieves a value in the integral lattice of $\mathfrak t^*$. Since $[\omega]$ is an integral class, then all the fixed points of $T$ achieve moment map values in the integral lattice of  $\mathfrak t^*$. By Example 6.10 in \cite{GGK}, the Lie algebra  $\mathfrak t$-action on $M$ determines a Lie group $T$-action on $M$. Hence we lifted the Hamiltonian $T$-action on the base manifold $N$ to the total space $M$ of the principal circle bundle $\pi\colon M\to N$. Now we have a $T\times S^1$-action on $M$, where the $S^1$-action  is the principal circle action on $M$ rotating the fibers of the bundle $\pi\colon M\to N$.

Let $\psi\colon M\to\mathfrak t^*\times\R^*$ be the contact moment map of the $T\times S^1$-action.
Then the image of $\psi$ lies in a hyperplane in $\mathfrak t^*\times\R^*$ given by
$$\psi^{\theta}=\langle \psi, \frac{\partial}{\partial\theta}\rangle= \alpha\Big(\frac{\partial}{\partial\theta}\Big)=1,$$ 
where we are using $\frac{\partial}{\partial\theta}$ both for the Reeb vector field of $\alpha$, the vector field generated by the principal $S^1$-action, and for the Lie algebra element in $\R=\,$Lie$(S^1)$. Up to a translation by a constant, this moment map image can be identified with the moment map image of the $T$-action on $N$, hence it is a convex polytope.

Recall that by Theorem~\ref{BW2}, we have a K-contact structure $(\alpha, g)$ on $M$, 
where $g = \pi^*(g_N)+ \alpha\otimes\alpha$, $g_N$ being a Riemannian metric on $N$.
We can check that for the lifted vector field $X_M$ in (\ref{lift}), we have
$$L_{X_M}\alpha = 0, \quad \forall \,\, X\in\mathfrak t,$$ 
so $\alpha$ is $T$-invariant, hence is $T\times S^1$-invariant. We can choose $g_N$ on $N$ to be $T$-invariant. Then $g$ on $M$ is $T\times S^1$-invariant. Hence the abelian Lie group $T\times S^1$ is in the isometry group of the K-contact manifold 
$(M, \alpha, g)$.

\smallskip

We summarize the main points of this section as follows.
\begin{theorem}\label{thmlift}
Let $(N, \omega)$ be a compact symplectic manifold of dimension $2n$ with symplectic class 
$[\omega]$ being integral, and $\pi\colon M\to N$ be  a principal circle bundle with connection form $\alpha$ such that $d\alpha=\pi^*(\omega)$. If there is a Hamiltonian torus $T$-action on $N$ with $1\leq \dim(T)\leq n$, then the $T$-action can be lifted to a $T$-action on $M$, and
the K-contact structure $(\alpha, g)$ on $M$ in Theorem~\ref{BW2} can be made to be $T\times S^1$-invariant, where $S^1$ is the principal $S^1$-action on $M$.
\end{theorem}

\section{changing the K-contact structures on the total space of the Boothby-Wang fibration $\pi\colon M\to N$ --- proof of Theorem~\ref{thm3}}

Let $(N, \omega)$ be a compact symplectic manifold with an integral symplectic class $[\omega]$.
By Theorem~\ref{BW2}, the total space of the principal circle bundle 
$\pi\colon M\to N$ with connection $\alpha$ such that $d\alpha=\pi^*(\omega)$ has  regular K-contact structures $(M, \alpha, g)$. By Theorem~\ref{thmlift}, if there is a Hamiltonian torus $T$-action on $N$, then the $T$-action can be lifted  to a $T$-action on $M$ such that the structure  $(\alpha, g)$ are $T\times S^1$-invariant. In this section, we show that we can perturb the contact form $\alpha$ to a new contact form $\alpha'$, and there exists a K-contact structure $(\alpha', g')$
 on $M$ such that the contact flow of $\alpha'$ has closure $T\times S^1$ in the isometry group of 
$(M, g')$. Moreover, we show that the closed Reeb orbits of $\alpha'$ are in one to one correspondence with the fixed points of the $T$-action on $N$. If $(N, \omega)$ is a K\"ahler manifold, then the K-contact structure $(\alpha', g')$ on $M$ is a Sasakian structure. Using this, we will give the proof of Theorem~\ref{thm3}.

\smallskip

First we have the following observation given by Yamazaki. 
\begin{proposition}\cite{Y}\label{Yama}
Let $(M, \alpha)$ be a compact $2n+1$-dimensional contact manifold, and $\varphi_t$ be the contact flow of $\alpha$. If there exists a torus $T$ with 
$1\leq\dim(T)\leq n+1$, a smooth effective $T$-action on $M$, and a homomorphism 
$h\colon \R\to T$ with dense image such that $\varphi_t = T\circ(h(t))$, then there exists an adapted Riemannian metric $g$ to $\alpha$ such that $(M, \alpha, g)$ is a K-contact manifold.
\end{proposition}
The idea of the proof of Proposition~\ref{Yama} is summarized below. Since the Reeb flow preserves $\alpha$, $\varphi_t = T\circ(h(t))$, and $h(t)$ is dense in $T$, the $T$-action preserves $\alpha$. Then the $T$-action preserves $d\alpha$, and in the subbundle $\ker (\alpha)$, we can choose $T$-invariant Riemannian metric $g_0$ and $T$-invariant almost complex structure $J$ such that 
$(g_0, J, d\alpha)$ are compatible. Then we have a $T$-invariant Riemannian metric $g=g_0+\alpha\otimes\alpha$ on $M$. In particular, the Reeb vector field is Killing for $g$. So $(M, \alpha, g)$ is a K-contact manifold.

\begin{proposition}\label{pert}
Let $(N, \omega)$ be a compact symplectic manifold  of dimension $2n$ with symplectic class 
$[\omega]$ being integral, and  $\pi\colon M\to N$ be a principal circle bundle with connection form $\alpha$ such that $d\alpha=\pi^*(\omega)$. Assume $(N, \omega)$ is equipped with a Hamiltonian torus $T$ action with $1\leq\dim(T)\leq n$. Then the following claims hold.
\begin{enumerate}
\item There exists a K-contact structure $(\alpha', g')$ on $M$ such that the closure of the Reeb flow of $\alpha'$ in the isometry group of $(M, g')$ is $T\times S^1$, where $S^1$ is the principal $S^1$-transformation group on $M$. The closed Reeb orbits of $\alpha'$ are in one to one correspondence with the fixed points of the $T$-action on $N$.
\item If $(N, \omega)$ is a K\"ahler manifold, then the K-contact structure in $(1)$ is Sasakian.
\end{enumerate} 
\end{proposition}

\begin{proof}
(1)
By Theorem~\ref{thmlift},  we have a $T\times S^1$-action on $M$ preserving $\alpha$, and the $1\times S^1$-action generates the Reeb vector field $\frac{\partial}{\partial\theta}$ of $\alpha$.
The contact moment map $\psi\colon M\to\mathfrak t^*\times \R^*$ of the $T\times S^1$-action on $M$ has image, which is a convex polytope, lying in the hyperplane in $\mathfrak t^*\times \R^*$ away from the origin given by
\begin{equation}\label{hy}
\psi^{\theta}=\langle \psi, \frac{\partial}{\partial\theta}\rangle= \alpha\Big(\frac{\partial}{\partial\theta}\Big)=1.
\end{equation}
Here we are using $\frac{\partial}{\partial\theta}$ both for the vector field on $M$ and for the Lie algebra element in $\R=\,$Lie$(S^1)$.
Since (\ref{hy}) holds and the image of $\psi$ is compact, we can find a generic element $\xi\in\mathfrak t\times\R=\,$Lie$(T\times S^1)$, an element whose one parameter subgroup $\exp(t\xi)$ is dense in $T\times S^1$, such that
$$\psi^{\xi} =\langle \psi, \xi\rangle = \alpha(\xi_M) > 0,$$
where $\xi_M$ is the vector field on $M$ generated by the $\exp(t\xi)$-action.
Let
$$\alpha' = \frac{\alpha}{\alpha(\xi_M)}.$$
Then $\alpha'$ is a new contact form on $M$ whose Reeb vector field  is $\xi_M$. The contact manifold $(M, \alpha')$ and the torus
$T\times S^1$ satisfy the condition of Proposition~\ref{Yama}, where $\dim(T\times S^1)\leq n+1$. By Proposition~\ref{Yama}, there exists a K-contact structure $(\alpha', g')$ on $M$ such that the closure of the Reeb flow of $\alpha'$ in the isometry group of $(M, g')$ is $T\times S^1$. 

Now we prove the second claim. Let $\phi\colon N\to\mathfrak t^*$ be the moment map of the $T$-action on $N$. Let $\xi = (\xi_1, \xi_2)\in\mathfrak t\times\R$ (we can take $\xi_2=\frac{\partial}{\partial\theta}$ if $\xi_1$ is suitably chosen). Up to a translation of the moment map $\phi$ by a constant, we may assume that $\phi^{\xi_1}\neq 0$.
First assume $n\in N$ is a fixed point for the $T$-action on $N$. Since $\xi_1\in\mathfrak t$ is generic, 
$n$ is a fixed point of $\exp(t\xi_1)$,  so ${\xi_1}_N(n)=0$. Then for any $m\in\pi^{-1}(n)$,   the horizontal lift $\Tilde{{\xi_1}_N} (m) = 0$. By (\ref{lift}),
$${\xi_1}_M(m) = 0 + \phi^{\xi_1}(n)\frac{\partial}{\partial\theta}(m).$$ 
So ${\xi_1}_M$ and ${\xi_2}_M$ are linearly dependent, both in the direction of $\frac{\partial}{\partial\theta}$, so the $\xi_M$-action through the point $m$ produces a closed Reeb orbit for the 1-form 
$\alpha'$. Conversely, if $n\in N$ is not a fixed point for the $T$-action on $N$, then 
${\xi_1}_N(n)\neq 0$ since $\xi_1\in\mathfrak t$ is generic. Then for any $m\in\pi^{-1}(n)$,
$${\xi_1}_M(m) = \Tilde{{\xi_1}_N} (m) + \phi^{\xi_1}(n)\frac{\partial}{\partial\theta}(m), \,\,\,\mbox{where}\,\,\, \Tilde{{\xi_1}_N} (m)\neq 0.$$
So ${\xi_1}_M$ and ${\xi_2}_M$ are linearly independent; so the closure of the Reeb orbit  through $m$ generated by $\xi_M$ is at least 2-dimensional since $\xi$ is generic, hence the Reeb orbit through $m$ is not closed.

(2) If $(N, \omega)$ is a K\"ahler manifold, by Theorem~\ref{BW2}, for the regular contact form, the connection form $\alpha$ on $M$, 
the K-contact structure $(\alpha, g)$ on $M$ is a Sasakian structure. By Theorem~\ref{thmlift}, this Sasakian structure is $T\times S^1$-invariant.  Hence the perturbed Reeb vector field $\xi_M$ of $\alpha'$ satisfies
$$L_{\xi_M} g = 0, \,\,\, \Big[\xi_M, \frac{\partial}{\partial\theta}\Big]=0,\, \,\, \alpha(\xi_M) > 0,$$
where  $\frac{\partial}{\partial\theta}$ represents the Reeb vector field of $\alpha$.
By \cite[Theorem 7.2]{YK}, the new K-contact structure $(\alpha', g')$ on $M$ in (1) is also Sasakian.
\end{proof}

\begin{lemma}\label{subgroup}
Let $N$ be a compact smooth manifold with a smooth torus $T^l$-action. Then for each $1\leq j\leq l$, there exists a subtorus $T^j\subset T^l$ action such that the fixed point set of the $T^j$-action is exactly the fixed point set of the $T^l$-action.
\end{lemma}

\begin{proof}
Since $N$ is compact, there exists an $S^1\subset T^l$ such that the $S^1$-action and the $T^l$-action have the same fixed point set. Take any subgroup of $T^l/S^1$ and identify it with a subgroup $T'\subset T^l$, then the $S^1\times T'$-action has the same fixed point set as the $S^1$-action, and as the $T^l$-action. 
\end{proof}

We can now prove Theorem~\ref{thm3}.

\begin{proof}[Proof of Theorem~\ref{thm3}]
By Lemma~\ref{subgroup}, for each $1\leq j\leq l$, there exists a subtorus $T^j\subset T^l$ action such that the fixed point set of the $T^j$-action is exactly the fixed point set of the $T^l$-action. 
Since the $T^l$-action is Hamiltonian, the $T^j$-action is also Hamiltonian.
For each fixed subgroup $T^j$, apply Proposition~\ref{pert} for the Hamiltonian $T^j$-action on $N$, we obtain the claims.
\end{proof}

For a better understanding of Proposition~\ref{pert}, we give a standard example the sphere $S^{2n+1}$ below. Some of the information appearing below is given in \cite{{R}, {Y}}.

\begin{example}\label{sphere}
Consider the unit sphere $S^{2n+1}$ 
$$\big\{(z_1, z_2, \cdots, z_{n+1})\in \C^{n+1} \,|\, |z_1|^2 + |z_2|^2 +\cdots + |z_{n+1}|^2 =1\big\}$$
with the standard contact form 
$$\alpha=\sum_i (x_i dy_i - y_i dx_i)$$ 
induced from the standard symplectic form $\omega = 2\sum_i dx_i\wedge dy_i$ on $\C^{n+1}$. The Reeb vector field of $\alpha$ is 
$$R = \sum_i\Big(x_i\frac{\partial}{\partial y_i} - y_i\frac{\partial}{\partial x_i}\Big).$$ 
This vector field is induced by the
$\R$-action, the Reeb flow of $\alpha$, given by
\begin{equation}\label{R}
t\cdot (z_1, z_2, \cdots, z_{n+1}) = (e^{it}z_1, e^{it} z_2, \cdots, e^{it} z_{n+1}), \quad \mbox{where}\,\, t\in\R.
\end{equation}
This flow generates a free $S^1$-action, hence the contact form $\alpha$ is regular.  We have an $S^1$-prinpical bundle $S^{2n+1}\to S^{2n+1}/S^1 \approx\CP^n$, the hopf fibration. A
K-contact metric adapted to $\alpha$ is 
$$g= g_0 + \alpha\otimes\alpha,$$ 
where 
$$g_0 = \sum_i \big(dz_i\otimes d\bar z_i \big)|_{\ker(\alpha)}=\sum_i\big(dr_i\otimes dr_i + r_i^2d\theta_i\otimes d\theta_i \big)|_{\ker(\alpha)}.$$

   Now we take $\lambda_1, \lambda_2, \cdots, \lambda_{n+1}$ to be rationally linearly independent positive real numbers, and take the vector field $R_{\lambda}$  given by 
$$R_{\lambda} =\sum_i\lambda_i \Big(x_i\frac{\partial}{\partial y_i} - y_i\frac{\partial}{\partial x_i}\Big).$$ 
Then we take a perturbed contact one-form of $\alpha$: 
\begin{equation}\label{nalpha}
\alpha' = \frac{\alpha}{\alpha(R_{\lambda})}= \sum_i \frac{(x_i dy_i - y_i dx_i)}{\sum_i (\lambda_i (x_i^2 + y_i^2))}.
\end{equation}
 The Reeb vector field of $\alpha'$ is $R_{\lambda}$, and $R_{\lambda}$ is generated by the $\R$-action on $S^{2n+1}$ given by
\begin{equation}\label{R'}
t\cdot (z_1, z_2, \cdots, z_{n+1}) = (e^{i\lambda_1t}z_1, e^{i\lambda_2t} z_2, \cdots, e^{i\lambda_{n+1}t} z_{n+1}), \quad \mbox{where}\,\, t\in\R.
\end{equation}
Note that there is a natural $T^{n+1}$-action on $S^{2n+1}$ induced from the $T^{n+1}$-action on $\C^{n+1}$. The
 $\R$-action (\ref{R'}) acts as a subgroup of $T^{n+1}$. The action in (\ref{R'}) has $n+1$ number of closed orbits, given by $(0, \cdots, e^{i\lambda_it}z_i, \cdots, 0)$, where $i=1, \cdots, n+1$. They correspond to the $n+1$ number of fixed points $[(0, \cdots, e^{i\lambda_it}z_i, \cdots, 0)]$
for the $T^{n+1}/S^1 = T^n$-action on $S^{2n+1}/S^1\approx \CP^n$, 
where the $S^1$-action on $S^{2n+1}$ is generated by the Reeb vector field $R$, or the action given by (\ref{R}).
Other Reeb orbits of $\alpha'$ are not closed. Note that the element $\xi = (\lambda_1, \lambda_2, \cdots, \lambda_{n+1})\in \R^{n+1}=\,$Lie$(T^{n+1})$ is generic, i.e., 
the subgroup $\exp(t\xi)$ is dense in $T^{n+1}$, since the $\lambda_i$'s are rationally linearly independent. 

By Proposition~\ref{Yama}, there exists a  K-contact metric adapted to $\alpha'$, it can be chosen to be
$$g' = g'_0 + \alpha'\otimes\alpha',\,\,\,\mbox{where}\,\,\, g'_0 = \frac{g_0}{\alpha(R_{\lambda})}|_{\ker(\alpha')}.$$

The K-contact structures $(\alpha, g)$ and $(\alpha', g')$ are Sasakian structures.
\end{example}

\begin{remark}
Similarly as in Example~\ref{sphere}, if  
$\xi = (\lambda_1, \lambda_2, \cdots, \lambda_{n+1})\in \R^{n+1}=\,$Lie$(T^{n+1})$ such that $\exp(t\xi)$ is dense in a $k$-dimensional torus $T^k$ of  $T^{n+1}$ with $k< n+1$, then by Proposition~\ref{Yama}, we have a  K-contact structure on $S^{2n+1}$ of rank $k$. The corresponding contact form then has non-isolated closed Reeb orbits. This is a different way of constructing K-contact structures of smaller rank from the construction in Theorem~\ref{thm3}, where the smaller dimensional torus has the same set of closed Reeb orbits as the bigger original torus.
\end{remark}

\section{examples of simply connected compact K-contact manifolds with minimal closed Reeb orbits which are not homeomorphic to the spheres --- proof of Theorem~\ref{thm4}}

In this section, we give a collection of examples of simply connected compact K-contact manifolds with minimal closed Reeb orbits that are not homeomorphic to the spheres. These examples are all real cohomology spheres.

\smallskip 

First, recall that by Theorem~\ref{thm2}, if $(N, \omega)$ is a compact symplectic manifold of dimension $2n$ equipped with a Hamiltonian $S^1$-action with $n+1$ isolated fixed points,
$[\omega]=x$ is primitive integral, and $H^*(N; \Z)=\Z[x]/x^{n+1}$, then the total space $M$ of the principal circle bundle over $N$ with first Chern class $x= [\omega]$ is homeomorphic to $S^{2n+1}$. 
By Theorem~\ref{thm3}, the manifold $M$ has a K-contact structure of rank $2$ which has exactly $n+1$ closed Reeb orbits. In particular, if 
$N=\CP^n$, then $M$ is $S^{2n+1}$. The sphere $S^{2n+1}$ has a standard regular K-contact (Sasakian) structure of rank 1. Since $\CP^n$ admits a Hamiltonian $T^n$-action with $n+1$ isolated fixed points, $S^{2n+1}$ admits a K-contact (Sasakian) structure of any rank $2\leq k\leq n+1$ which has $n+1$ number of closed Reeb orbits by Theorem~\ref{thm3}. 

Next, using Theorems~\ref{thm2} and \ref{thm3}, we will get a collection of compact K-contact manifolds of dimension $2n+1$ with $n+1$ number of closed Reeb orbits which are not homeomorphic to 
$S^{2n+1}$.  

First, in Example~\ref{exgrass}, we have a family of examples for any $n\geq 3$ odd. 

\begin{example}\label{exgrass}
Let $(N, \omega)$ be a compact symplectic manifold of dimension $2n$ admitting a Hamiltonian $S^1$-action with exactly $n+1$ isolated fixed points, where $n\geq 3$ is odd, and $x=[\omega]$ is primitive integral. Assume the integral cohomology ring $H^*(N; \Z)$ is generated by
$$1, \,x, \,\cdots,\, x^{\frac{n-1}{2}},\, \frac{1}{2}x^{\frac{n+1}{2}},\, \cdots,\, \frac{1}{2}x^{n-1},\, \frac{1}{2}x^n.$$
For example, $\Gt_2(\R^{n+2})$, the Grassmannian of oriented two planes in $\R^{n+2}$ for any $n\geq 3$ odd, can be such a manifold $N$, see \cite{L}; since $\Gt_2(\R^{n+2})$ is a coadjoint orbit of 
$SO(n+2)$, it is a K\"ahler manifold. Note that $N$ may be different from $\Gt_2(\R^{n+2})$ as a manifold.
Let $\pi\colon M\to N$ be the principal $S^1$-bundle over $N$ with first Chern class $x$. By Theorem~\ref{thm2} (1), $M$ is not homeomorphic to $S^{2n+1}$. By Theorem~\ref{thm1} (1), $M$ has the real cohomology ring of $S^{2n+1}$. By Lemma~\ref{pi1m} (2), 
$$\pi_1(M)=1.$$

To get the precise integral cohomology groups (ring) of $M$, we can
use the Gysin exact sequence for the circle bundle $\pi\colon M\to N$:
$$\cdots\to H^*(M; \Z)\to H^{*-1}(N; \Z)\to H^{*+1}(N; \Z)\to H^{*+1}(M; \Z)\to\cdots,$$
and get
 \[  H^*(M; \Z) = \left\{ \begin{array}{ll}
           \Z    &  \mbox{if $ * = 0, 2n+1$},\\
           \Z_2  &    \mbox{if $* =n+1$},\\
           0    & \mbox{other $*$'s}.
           \end{array} \right. \]

By Theorem~\ref{thm3}, the simply connected manifold $M$ admits a K-contact structure of rank $2$ with exactly $n+1$ closed Reeb orbits.  In particular, if $N=\Gt_2(\R^{n+2})$, then 
$M=V_2(\R^{n+2})=SO(n+2)/SO(n)=O(n+2)/O(n)$, the Stiefel manifold of orthonormal 2-frames in $\R^{n+2}$. The maximal torus of $SO(n+2)$ for $n\geq 3$ odd is $T^{\frac{n+1}{2}}$. As a coadjoint orbit of $SO(n+2)$,  $\Gt_2(\R^{n+2})$ admits a Hamiltonian $T^{\frac{n+1}{2}}$-action with $n+1$ isolated fixed points (see \cite{Mo}), by Theorem~\ref{thm3}, we can have K-contact structures of different ranks on 
$V_2(\R^{n+2})$ which have exactly $n+1$ closed Reeb orbits.  Since $\Gt_2(\R^{n+2})$ is K\"ahler, 
 these K-contact structures on $V_2(\R^{n+2})$ are Sasakian.
\end{example}

\begin{example}\label{exg2}
Let $(N, \omega)$ be a compact symplectic manifold of dimension $10$ admitting a Hamiltonian $S^1$-action with exactly 6 isolated fixed points, where $x=[\omega]$ is primitive integral. Assume the integral cohomology ring $H^*(N; \Z)$ is generated by
$$1, \, x, \, \frac{1}{3}x^2, \, \frac{1}{6}x^3, \, \frac{1}{18}x^4, \, \frac{1}{18}x^5.$$
For example, a 10-dimensional coadjoint orbit of the Lie group $G_2$, denoted $\mathcal O$, is such a manifold,  see \cite{L'}. As a coadjoint orbit of $G_2$,
$\mathcal O$ is a K\"ahler manifold. Note that we may have other such manifolds than $\mathcal O$.
Let $\pi\colon M\to N$ be the principal circle bundle over $N$ with first Chern class $x$.
By Theorem~\ref{thm2} (1), $M$ is not homeomorphic to $S^{11}$. By Theorem~\ref{thm1} (1), $M$ has the real cohomology ring of $S^{11}$.  By Lemma~\ref{pi1m} (2), 
$$\pi_1(M)=1.$$

By the Gysin exact sequence for the circle bundle $\pi\colon M\to N$,
we can get the precise integral cohomology groups (ring) of $M$: 
 \[  H^*(M; \Z) = \left\{ \begin{array}{ll}
           \Z    &  \mbox{if $ * = 0, 11$},\\
           \Z_3  &    \mbox{if $* =4$},\\
           \Z_2  &    \mbox{if $* =6$},\\
           \Z_3  &    \mbox{if $*=8$},\\
           0    & \mbox{other $*$'s}.
           \end{array} \right. \]

By Theorem~\ref{thm3}, the simply connected $11$-dimensional manifold $M$ admits a K-contact structure of rank 2 with exactly $6$ closed Reeb orbits.  In particular, if $N=\mathcal O$, 
it can be written as a quotient by $U(2)$, i.e., $\mathcal O = G_2/U(2)$. The Lie group $G_2$ is simply connected; by \cite[Theorems 2 and 11]{K}, the principal circle bundles over $\mathcal O =G_2/U(2)$, as a group, is generated by $G_2/SU(2)$, where $SU(2)$ is the semi-simple part of $U(2)$. Hence  the principal circle bundle over $\mathcal O = G_2/U(2)$ with
first Chern class $x$, the generator of $H^2(\mathcal O, \Z)$, is $M=G_2/SU(2)$. By Theorem~\ref{thm3}, the rank 2 K-contact structure on $M=G_2/SU(2)$ is Sasakian. 
Moreover, the maximal torus $T^2$ of $G_2$ acts on $\mathcal O = G_2/U(2)$ in a Hamiltonian way also with $6$ isolated fixed points, see \cite{L'}. Theorem~\ref{thm3} gives a K-contact (Sasakian) structure of rank 3 on $M=G_2/SU(2)$ with exactly $6$ isolated closed Reeb orbits.
\end{example}

\begin{example}\label{exv5}
Let $(N, \omega)$ be a compact symplectic manifold of dimension $6$ admitting a Hamiltonian $S^1$-action with exactly $4$ isolated fixed points, where $x=[\omega]$ is primitive integral. Assume the integral cohomology ring $H^*(N; \Z)$ is 
generated by
$$1, \,\,x,\, \,\frac{1}{5}x^2,\,\, \frac{1}{5}x^3.$$
For example, $V_5$, a fano manifold, is such a manifold. The fact that a compact $6$-dimensional Hamiltonian $S^1$-manifold with exactly 4 isolated fixed points can have the above integral cohomology ring structure is discovered by Tolman in \cite{T}. In \cite{Mc}, McDuff gives $V_5$ as such an example, and shows that $V_5$ is the only such manifold up to $S^1$-equivariant symplectomorphism;  moreover, $V_5$ is an $S^1$-invariant K\"ahler manifold.
Let $\pi\colon M\to N$ be the principal circle bundle over $N$ with first Chern class $x$. 
By Theorem~\ref{thm2} (1), $M$ is not homeomorphic to $S^7$. By Theorem~\ref{thm1} (1), $M$ has the real cohomology ring of $S^7$. By Lemma~\ref{pi1m} (2), 
$$\pi_1(M)=1.$$

The Gysin exact sequence for the circle bundle $\pi\colon M\to N$ gives the precise integral cohomology groups (ring) of $M$: 
 \[  H^*(M; \Z) = \left\{ \begin{array}{ll}
           \Z    &  \mbox{if $ * = 0, 7$},\\
           \Z_5  &    \mbox{if $* =4$},\\
           0    & \mbox{other $*$'s}.
           \end{array} \right. \]

By Theorem~\ref{thm3},  the simply connected $7$-dimensional manifold $M$ admits a K-contact structure of rank 2 with exactly $4$ closed Reeb orbits. Due to the uniqueness of $N=V_5$, the K-contact structure on $M$ is a Sasakian structure.
\end{example}

\begin{example}\label{exv22}
Let $(N, \omega)$ be a compact symplectic manifold of dimension $6$ admitting a Hamiltonian $S^1$-action with eactly $4$ isolated fixed points, where $x=[\omega]$ is primitive integral. Assume the integral cohomology ring  $H^*(N; \Z)$ is generated by
$$1, \,\, x, \,\, \frac{1}{22}x^2, \,\,\frac{1}{22}x^3.$$
For example, $V_{22}$, a fano manifold, is such a manifold. Similar to Example~\ref{exv5}, 
$V_{22}$ is the unique such manifold up to $S^1$-equivariant symplectomorphism, and it is an $S^1$-invariant K\"ahler manifold, see \cite{T, Mc}. Let $\pi\colon M\to N$ be the principal circle bundle over $N$ with first Chern class $x$. By Theorem~\ref{thm2} (1), $M$ is not homeomorphic to $S^7$. By Theorem~\ref{thm1} (1), $M$ has the real cohomology ring of $S^7$. By Lemma~\ref{pi1m} (2), 
$$\pi_1(M)=1.$$

The Gysin exact sequence for the circle bundle $\pi\colon M\to N$ gives the integral cohomology groups (ring) of $M$: 
 \[  H^*(M; \Z) = \left\{ \begin{array}{ll}
           \Z    &  \mbox{if $ * = 0, 7$},\\
           \Z_{22}  &    \mbox{if $* =4$},\\
           0    & \mbox{other $*$'s}.
           \end{array} \right. \]

By Theorem~\ref{thm3}, the simply connected $7$-dimensional manifold $M$ admits a K-contact structure of rank $2$ which has exactly $4$ closed Reeb orbits. Due to the uniqueness of $N=V_{22}$, the K-contact structure on $M$ is a Sasakian structure.
\end{example}

For the examples in Example~\ref{exv5} and \ref{exv22}, we do not know if the 7-dimensional contact manifolds, the total spaces of the circle bundles over $V_5$ and $V_{22}$ are known and if they have names.

\section{examples of simply connected K-contact manifolds which are not real cohomology spheres}

In this section, through the Boothby-Wang fibration, we give a family of examples of simply connected compact K-contact manifolds of dimension $2n+1$ which are not real cohomology spheres, and admit K-contact structures with  more than $n+1$ number of isolated closed Reeb orbits, by our Theorems~\ref{thm1} and \ref{thm3}.

\begin{example}\label{evengrass}
Let $(N, \omega)$ be a compact symplectic manifold of dimension $2n$ admitting a Hamiltonian $S^1$-action with exactly $n+2$ isolated fixed points, where $x=[\omega]$ is primitive integral. By \cite{L21}, $n$ must be even. For example, $\Gt_2(\R^{n+2})$, the Grassmannian of oriented two planes in $\R^{n+2}$ with $n > 2$ even, is such a manifold, and $\Gt_2(\R^{n+2})$ is a K\"ahler manifold as 
a coadjoint orbit of $SO(n+2)$,  see \cite{L21}.
Let $\pi\colon M\to N$ be a principal circle bundle over $N$ with first Chern class $x$. By Theorem~\ref{BW2}, $(M, \alpha)$ is a K-contact manifold, where $\alpha$ is a connection one form on $M$ such that $\pi^*(\omega) = d\alpha$. 
By \cite{L21}, the manifold $N$ has odd Betti numbers all zero, and even Betti numbers as follows:
$$b_{2i}(N)=1\,\,\, \mbox{for each $0\leq 2i\leq 2n$ and $2i\neq n$,  and} \,\,\, b_n(N) = 2.$$
By Theorem~\ref{thm1} (1), $M$ is not a real cohomology sphere. We can use
 the Gysin exact sequence  in $\R$-coefficients
$$\cdots\to H^*(M; \R)\to H^{*-1}(N; \R)\to H^{*+1}(N; \R)\to H^{*+1}(M; \R)\to\cdots$$
to get the non-zero Betti numbers of $M$:
$$b_0 (M) = b_n(M) = b_{n+1}(M) = b_{2n+1}(M) = 1.$$
(We can also use the Gysin exact sequence in $\Z$-coefficients to obtain the integral cohomology of $M$.)
We still have $H^2(N; \Z) = \Z$ (\cite{L21}),  by Lemma~\ref{pi1m} (2), 
$$\pi_1(M)=1.$$
By Theorem~\ref{thm3}, the simply connected manifold $M$ admits a K-contact structure of rank 2 with $n+2$ closed Reeb orbits.  In particular, if $N=\Gt_2(\R^{n+2})$ with $n > 2$ even, then 
$M=V_2(\R^{n+2})=SO(n+2)/SO(n)=O(n+2)/O(n)$, the Stiefel manifold of orthonormal 2-frames in $\R^{n+2}$. The maximal torus of $SO(n+2)$ with $n > 2$ even is $T^{\frac{n}{2}+1}$. The  K\"ahler manifold $\Gt_2(\R^{n+2})$ admits a Hamiltonian $T^{\frac{n}{2}+1}$-action with $n+2$ isolated fixed points (\cite{Mo}). By Theorem~\ref{thm3}, we have K-contact (Sasakian) structures on $V_2(\R^{n+2})$ of different ranks with $n+2$ closed Reeb orbits.  
\end{example}

\end{document}